\documentclass[a4paper]{article}

\usepackage[T1]{fontenc}

\usepackage{amsmath}
\usepackage{amsfonts}
\usepackage{amssymb}
\usepackage{amsthm}

\usepackage{ifpdf}
\usepackage{ifthen}

\usepackage{enumerate}

\usepackage{aliascnt}

\ifpdf
\usepackage[colorlinks=true]{hyperref}
\else
\usepackage[draft=true]{hyperref}
\fi
\usepackage{cleveref}

\usepackage[all]{xy}

\usepackage{pifont}

\ifpdf
\usepackage[pdftex]{graphicx,color}
\else
\usepackage{graphicx}
\fi

\usepackage{setspace}

\usepackage{footnote}

\usepackage{sectsty}
\sectionfont{\normalsize}

\swapnumbers

\newtheorem{thm}{Theorem}[section]
\crefname{thm}{theorem}{theorems}
\Crefname{thm}{Theorem}{Theorems}

\newaliascnt{lemma}{thm}
\newtheorem{lemma}[lemma]{Lemma}
\aliascntresetthe{lemma}

\newaliascnt{prop}{thm}

\aliascntresetthe{prop}
\crefname{prop}{proposition}{propositions}
\Crefname{prop}{Proposition}{Propositions}

\newaliascnt{cor}{thm}

\aliascntresetthe{cor}
\crefname{cor}{corollary}{corollaries}
\Crefname{cor}{Corollary}{Corollaries}

\theoremstyle{definition}

\newaliascnt{rem}{thm}
\newtheorem{rem}[rem]{Remark}
\aliascntresetthe{rem}
\crefname{rem}{remark}{remarks}
\Crefname{rem}{Remark}{Remarks}

\newaliascnt{example}{thm}
\newtheorem{example}[example]{Example}
\aliascntresetthe{example}

\newaliascnt{definition}{thm}
\newtheorem{definition}[definition]{Definition}
\aliascntresetthe{definition}

\newaliascnt{war}{thm}

\aliascntresetthe{war}
\crefname{war}{warning}{warnings}
\Crefname{war}{Warning}{Warnings}

\newaliascnt{block}{thm}
\newtheorem{block}[block]{}
\aliascntresetthe{block}
\crefname{block}{}{}
\Crefname{block}{}{}

\newaliascnt{conv}{thm}

\aliascntresetthe{conv}
\crefname{conv}{convention}{conventions}
\Crefname{conv}{Convention}{Conventions}

\newaliascnt{notation}{thm}

\aliascntresetthe{notation}
\crefname{notation}{notation}{notations}
\Crefname{notation}{Notation}{Notations}

\newaliascnt{notconv}{thm}
\newtheorem{notconv}[notconv]{Notation and conventions}
\aliascntresetthe{notconv}
\crefname{notconv}{}{}
\Crefname{notconv}{}{}

\newtheorem*{ack}{Acknowledgements}

\numberwithin{equation}{section}
\crefformat{equation}{eq.~#2#1#3}
\Crefformat{equation}{Eq.~#2#1#3}

\crefname{subsection}{subsection}{subsections}
\Crefname{subsection}{Subsection}{Subsections}

\crefformat{block}{#2#1#3}
\Crefformat{block}{#2#1#3}

\crefformat{enumi}{#2#1#3}
\crefrangeformat{enumi}{(#3#1#4-#5#2#6)}

\newcommand{\C}{\mathbb{C}}

\newcommand{\N}{\mathbb{N}}
\newcommand{\Z}{\mathbb{Z}}
\newcommand{\R}{\mathbb{R}}

\newcommand{\W}{\mathcal{W}}
\newcommand{\V}{\mathcal{V}}
\newcommand{\A}{\mathcal{A}}
\newcommand{\E}{\mathcal{E}}

\newcommand{\X}{\tilde X}

\newcommand{\set}[2]{\left\{ #1 \,\middle\vert\, #2 \right\}}

\renewcommand{\epsilon}{\varepsilon}

\newcommand{\mynd}[4]
{
    \begin{figure}[#2]\begin{center}
        \ifpdf
            \input{#1.pdf_t}
        \else
            \input{#1.pstex_t}
        \fi
		\ifthenelse{\equal{#3}{}}
        {
		}
		{
            \caption{#3\label{#4}}
        }
    \end{center}
    \end{figure}
}

\newcounter{dummy}
\renewcommand{\thedummy}{\roman{dummy}}
\crefformat{dummy}{#2(#1)#3}
\crefrangeformat{dummy}{(#3#1#4-#5#2#6)}

\newenvironment{blist}
{
  \begin{list}{(\thedummy)}
  {
	\setlength\labelsep{4pt}
	\setlength\itemindent{4pt}
	\setlength\leftmargin{0pt}
	\setlength\labelwidth{0pt}
    \usecounter{dummy}
  }
}
{
  \end{list}
}

\hypersetup{colorlinks=false}

\author{Baldur Sigur\dh sson}
\title{The boundary of the Milnor fiber of the singularity
$f(x,y) + zg(x,y) = 0$.}
\date{\today}

\begin{document}

\begin{center}
{ \bf \large
The boundary of the Milnor fiber of the singularity $f(x,y) + zg(x,y) = 0$
}
\vspace{0.3cm}

{Baldur Sigur\dh sson
\footnote{Baldur Sigur\dh sson, Basque Center for Applied Mathematics,
Bilbao, Spain; \tt{bsigurdsson@bcamath.org}}
}
\end{center}



\begin{abstract}
Let $f,g\in\C\{x,y\}$ be germs of functions defining plane curve
singularities without common components
in $(\C^2,0)$ and let $\Phi(x,y,z) = f(x,y) + zg(x,y)$. We give an
explicit algorithm producing a plumbing graph for the boundary
of the Milnor fiber of $\Phi$ in terms of a common resolution for
$f$ and $g$.
\end{abstract}

\section{Introduction}

It is known that the boundary of any hypersurface singularity in
$(\C^3,0)$ is a plumbed manifold.
This was stated by Michel and Pichon in \cite{Mich_Pich} and proved
by separate methods by N\'emethi and Szil\'ard \cite{Nem_Szil} and
Michel and Pichon \cite{Mich_Pich_car}.
A stronger statement for certain
real analytic map germs was proved by de Bobadilla and Neto \cite{Boba_Neto}.
As these theorems rely on resolution of singularities, they do
not easily provide an explicit description of a plumbing graph describing
the boundary. Calculations have been carried out, however for some
particular singularities and families:
for Hirzebruch singularities \cite{Mich_Pich_Web_Hirz}, suspensions
of isolated plane curves \cite{Mich_Pich_Web_susp} and in the
many examples of \cite{Nem_Szil}.

In the case of a hypersurface singularity given by the equation
\[
  \Phi(x,y,z) = f(x,y) + zg(x,y) = 0,
\]
where $f,g$ are singular germs with
no common factors (but not necessarily reduced), we give an explicit
algorithm producing a plumbing graph for the boundary of the Milnor fiber
in terms of the graph associated with an embedded resolution of
the plane curve singularities defined by
$f$ and $g$. This is obtained from an explicit
description of the Milnor fiber by the author \cite{Sigurdsson2016}.
Singularities of the form $f(x,y) + zg(x,y)$ are closely related to
the deformation theory of sandwitched singularities, see
\cite{SANDWICH}.
The article is organized as follows.

In \cref{s:Milnor_fiber} we recall the results of \cite{Sigurdsson2016}
and fix notation related to the resolution graph of $f$ and $g$.

In \cref{s:plumbing} we define plumbed manifolds and prove some
useful lemmas related to them.

In \cref{s:mult_def} we introduce families of multiplicities
and dual multiplicities assigned to a complex valued function on
a plumbed $3$-manifold, satisfying certain conditions. In the case
of a fibration over $S^1$, these multiplicities coincide with the
multiplicities used in \cite{EisNeu,Nem_Szil}.

In \cref{s:cont_frac} we prove a useful lemma relating the negative
continued fraction expansion of a rational number to a plumbing
construction.

In \cref{s:constr} and \cref{s:mult} we provide the details of the construction
of the plumbing graph for the boundary of the Milnor fiber of
$\Phi$ and the families of multiplicities and dual multiplicities
for the coordinate function $z$. These statements can be read
after only reading \cref{def:Gamma}.

In \cref{s:examples} we provide some examples. First we give the
simple plumbing graph describing boundary of the Milnor fiber of
a $T_{a,b,*}$ singularity given by the equation $x^a + y^b + xyz = 0$.
This example is discussed in \cite{Nem_Szil} 22.2.

\Cref{s:proofs} contains proofs of \cref{thm:boundary} and \cref{thm:mult}.

\begin{notconv} \label{notconv:notation}
\begin{blist}
\item \label{it:notation_disks}
We denote by $D\subset \C$ the open unit disk and by $\overline D$
the closed unit disk. We also set $S^1 = \partial \overline D$.
For any $r>0$, let $D_r,\overline D_r, S^1_r$ be the corresponding disks
and circle with radius $r$.

\item \label{it:notation_fclass}
If $X$ is a manifold, and $C\subset X$ is a submanifold of dimension
$d$, then we denote by $[C]\in H_d(M,\Z)$ the associated homology
class.
If $X$ is a compact oriented compact manifold, possibly with
boundary, we denote by $(\cdot,\cdot)_X$ the intersection pairing
between $H_i(X,\Z)$ and $H_{n-i}(X,\partial X,\Z)$, where
$n = \dim X$.
In particular, if $\partial X = \emptyset$ and $i = n/2$,
then $(\cdot,\cdot)_X$ is the intersection form on the middle
homology.

\item \label{it:notation_orientation}
The boundary of an oriented manifold is oriented by the usual
\emph{outward-pointing-vector first} rule. Note that if a codimension
one submanifold $N\subset M$
splits $M$ into two pieces, this rule induces
opposite orientations according to which piece is chosen.

\item \label{it:notation_bundles}
A locally trivial differential fiber bundle with a chosen orientation
on the total space and the base space induces an orientation on each fiber
by the following requirement. A lifting of a positive basis of the tangent
space of the base space, followed by a positive basis of the tangent
space of the fiber yields a positive orientation of the total space.
In fact, this rule induces an orientation on the fibers, the total
space or the base space, given orientations on the other two.

\end{blist}
\end{notconv}

\begin{ack}
I would like to thank N\'emethi Andr\'as for suggesting this problem to
me and for the many helpful discussions we have had.
\end{ack}

\section{The Milnor fiber} \label{s:Milnor_fiber}

In \cite{Sigurdsson2016} the author gives a description of the Milnor fiber of
the singularity $f(x,y) + z g(x,y) = 0$. We will now recall that result
and fix some notation.

\begin{definition} \label{def:Gamma}
Let $\phi:V\to\C^2$ be a common resolution of the functions $f$ and $g$
with exceptional divisor $E$, decomposing into irreducible components
as $E = \cup_{v\in\W} E_v$ and
denote by $\Gamma$ the associated embedded resolution graph.
The set of vertices in $\Gamma$
is $\V = \W\amalg \A$, where $\W$ corresponds to components of the
exceptional divisor, while elements of $\A$ are \emph{arrowheads},
corresponding to components of the strict transforms of $f$ and $g$.
For any $a\in\A$ there is a $w_a\in\W$
so that $\{w_a,a\}$ is an edge in $\Gamma$.
Write $\A = \A_f \amalg \A_g$, where elements of
$\A_f$ and $\A_g$ correspond to components of the strict transform
of $f$ and $g$.
For $v\in\V$, we denote by $m_v$ and $l_v$ the multiplicities of $f$ and
$g$, respectively. In particular, $m_v = 0$ if and only if $v\in\A_g$
and, similarly, $l_v = 0$ if and only if $v\in\A_f$.

Define $\W_1 = \set{w\in\W}{m_w\leq l_w}$ and
$\W_2 = \set{w\in\W}{m_w> l_w}$.
 Write
$\A_i = \set{a\in\A}{w_a\in\W_i}$ for $i=1,2$. Similarly, take
$\A_{f,i},\A_{g,i}\subset\A_i$ so that $\A_f = \A_{f,1}\amalg \A_{f,1}$
and $\A_g = \A_{g,1}\amalg \A_{g,2}$.
\end{definition}

\begin{definition} \label{def:neighbourhoods}
For $w\in\W$, let $T_w$ be a tubular neighbourhood around $E_w$ in $V$
and let $T = \cup_{w\in\W} T_w$. Set also $T_i = \cup_{w\in\W_i} T_w$
for $i=1,2$.
For a given $0 < \epsilon \ll 1$, let
$F_f = f^{-1}(\epsilon)$ be the Milnor fiber of $f$, and
$F'_f = \phi^{-1}(F_f)$ its pullback to $V$. Let $T_\epsilon$ be a small
tubular neighbourhood around $F'_f$ in $T$.
We also choose tubular neighbourhoods $T_a \subset T$ around $E_a$ for
any $a\in\A$.
With these choices
fixed, choose a small tubular neighbourhood $T' = \cup_{w\in\W} T'_w$
around the exceptional divisor inside $T$. This is chosen small enough that
$T' \cap T_\epsilon = \emptyset$.

Set $T_{f,i} = \cup_{a\in\A_{f,i}} T_a$, and
$T_{g,i} = \cup_{a\in\A_{g,i}} T_a$ for $i=1,2$ and let
\[
  \overline T_{f,g}
    = (\overline T_{f,1} \setminus T')
    \cup \overline T_\epsilon
    \cup (\overline T_2 \setminus (T'\cup T_{g,2})).
\]
\end{definition}

\mynd{scheme}{ht}
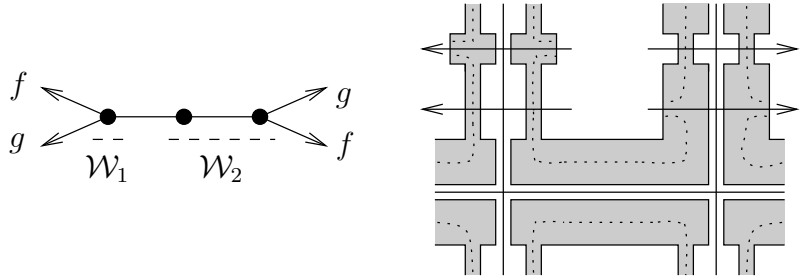
{A schematic picture showing $\overline T_{f,g}$. The Milnor fiber of $f$
is shown as a dotted curve. The Milnor fiber of $\Phi$ is obtained by
twisting along the strict transform of $g$.}
{im:scheme} 

\begin{definition}
Let $X$ be a four dimensional manifold with boundary and
$\iota: \overline D \to X$ an embedding of the closed disk into $X$
such that $\iota$ sends $S^1=\partial \overline D$ to $\partial X$
and the image of $\overline D$ is transversal to $\partial X$.
Then there exists a map $\psi:\overline D\times \overline D \to X$
parametrizing a tubular neighbourhood of $\iota (D)$ in $X$
so that $\psi(0,z) = \iota(z)$ for $z\in\overline D$ and
$\psi(x,z) \in\partial X$ for $x\in\overline D$ and
$z\in S_1 = \partial D$.
For $k\in\Z$, the \emph{$k^{\mathrm{th}}$ twist} along
$\iota$ is defined as
$(X\setminus\psi(D\times\overline D))\amalg_{t_k}\overline D\times \overline D$
where the glueing map
$t_k:S^1\times \overline D \to (X\setminus\psi(D\times\overline D))$ is
defined by $t_k(x,z) = \psi(x,x^k z)$ and is
denoted by $X_{\iota,k}$.
We also say that $X_{\iota,k}$ is obtained from $X$ by \emph{twisting
$X$ $k$ times along} $\iota(\overline D)$.
\end{definition}

\begin{definition} \label{def:T_f_g}
In \cite{Sigurdsson2016},
the author shows that for any $a\in A_g$, the intersection
$E_a\cap \overline T_{f,g}$ is a disjoint union of $m_{w_a}$ disks embedded
in $T_{f,g}$. Let $F_{f,g}$ be the manifold obtained from $\overline T_{f,g}$
by twisting each of these disks $l_a$ times for all $a$.
\end{definition}

\begin{thm}[\cite{Sigurdsson2016}] \label{thm:whole_fiber}
The Milnor fiber $F_\Phi$ is diffeomorphic to $F_{f,g}$. \qed
\end{thm}

\begin{definition} \label{def:M}
Let $M_{f,g} = \partial F_{f,g}$. We also set
$M'_{f,g} = \partial \overline T_{f,g}$.
\end{definition}

\section{Plumbed 3-manifolds} \label{s:plumbing}

In this section we give an introduction to plumbed three manifods
and plumbing graphs, along with some useful properties. Throughout
this text, an $S^1$-bundle will mean a principal $S^1$-bundle. In particular,
we assume that there is a consistent choice for orientation on each fiber.
In fact, all our $S^1$-bundles will have as base space an \emph{oriented}
real surface. This determines a consistent choice of orientation on
fibers as described in \cref{notconv:notation}\cref{it:notation_orientation}.

We note that apart from our restriction on orientability, our definition
of a plumbed manifold is equivalent to the definition in \cite{Neu_plumb}.
This can be seen from \cref{lem:edge_signs}. We note, however, that our
construction differs slightly to the standard one. This is
explicated in \cref{rem:diff}.
The main reason for this is that in our construction in
\cref{s:proofs}, we identify the three dimensional plumbed pieces directly,
but the result can in no natural way be seen as the boundary
of a four dimensional plumbed manifold (as is the case for links
of isolated surface singularities).

\begin{definition} \label{def:plumbing}
A \emph{plumbed manifold} is a three dimensional compact manifold $M$,
possibly with boundary,
given as a union of submanifolds with boundary
$M = \cup_{v\in\W} M_v$ having the following properties.
\begin{enumerate}

\item
For each $v,w\in\W$, $v\neq w$ we have an $e_{v,w}\in\N$ so that
\[
  M_v\cap M_w = \amalg_{i=1}^{e_{v,w}} S_{v,w,i},
\]
with $S_{v,w,i}$ an embedded torus $M \supset S_{v,w,i} \cong S^1\times S^1$.
Thus, $S_{v,w,i}$ is a component of $\partial M_v$ and inherits an orientation.
Since $e_{v,w} = e_{v,w}$, we can assume that as sets, we have 
$S_{w,v,i} = S_{v,w,i}$ for $i=1, \dots, e_{v,w}$.

\item
For each $v$ we have a compact connected
surface $\Sigma_v$ (possibly with boundary) and a locally trivial
$S^1$ bundle $\pi_v:M_v\to\Sigma_v$. If $1\leq i \leq e_{v,w}$ for some
$w\neq v$, then $S_{v,w,i} = \pi^{-1}(B_{v,w,i})$ where
$B_{v,w,i} \cong S^1$ is a component of the boundary of $\Sigma$.

\item
Assume that $1\leq i \leq e_{v,w}$ for some $v\neq w$. The map
\[
  S_{v,w,i} \to B_{v,w,i} \times B_{w,v,i}, \quad
  p\mapsto (\pi_v(p), \pi_w(p))
\]
is a diffeomorphism.

\item
For each $v\in\V$, let $B_{v,1}, \ldots, B_{v,e_v}$ be the components
of $\partial \Sigma$ not of the form $B_{v,w,i}$ for some $w,i$. We assume
given a section $s_{v,i}:B_{v,i} \to \Sigma$ to the reduced
bundle $\pi_v|_{B_{v,i}}$.

\end{enumerate}
\end{definition}

\begin{block}
We orient $S_{v,w,i}$ by considering it as a subset of the boundary
of $M_v$.
This way, $S_{v,w,i} = S_{w,v,i}$ as sets, but
$S_{v,w,i} = -S_{w,v,i}$ as oriented manifolds.
We also orient the boundary of $\Sigma_w$ by the same rule, for any $w\in\W$.
\end{block}

\begin{block}
For a closed surface $\Sigma$, the \emph{Euler number} classifies
the $S^1$ bundles over $\Sigma$.
However, every $S^1$ bundle $\pi:M\to \Sigma$ over
a compact surface with nonempty boundary is trivial.
But given a trivialization, or, equivalently, a section
 $s:\partial \Sigma \to \partial M$, over the
boundary, a \emph{relative Euler number} is well defined, and invariant
under homotopy of the section.
This is a complete invariant in the following sense. Let 
$\Sigma$ be a compact surface with boundary and take two
$S^1$ bundles $M,M'\to \Sigma$ with sections $s:\partial\Sigma \to M$
and $s':\partial\Sigma \to M'$ and an isomorphism of bundles
$\psi:M|_{\partial\Sigma} \to M'|_{\partial\Sigma}$ sending $s$ to $s'$.
Then $\psi$ extends to an isomorphism of bundles $M\to M'$ if and only
if the relative Euler numbers coincide. We will refer to the relative
Euler number simply as the Euler number.

The relative Euler number is defined as follows. Let $D \subset \Sigma$ be
an open disk. We can extend the section $s:\partial \Sigma \to \partial M$
to a section $\overline s:\Sigma \setminus D \to M \setminus \pi^{-1}(D)$.
Given an orientation preserving diffeomorphism $\varphi:\partial D \to S^1$,
there is a unique number $b\in\Z$ so that the twisted section
$\overline s^b:\partial D \to \pi^{-1}(\partial D)$,
$x \mapsto \varphi(x)^b \overline s(x)$ extends over the disk $D$.
The relative Euler number is defined as $-b$.
\end{block}

\begin{lemma} \label{lem:triv}
Let $M\to \Sigma$ be an $S^1$ bundle over a compact surface with
boundary. Let $-b$ be its the Euler number relative to a section
$s:\partial\Sigma\to M$. Let $C\subset M$ be a fiber of the bundle
and $C'$ the image of $s$ (as \emph{oriented} submanifolds). Then, in
$H_1(M,\Z)$
\begin{equation} \label{eq:lem_triv}
   -b[C] = [C'].
\end{equation}
\end{lemma}
\begin{proof}
This follows from the definition of the relative Euler number. Indeed, let
$\overline s :\Sigma\setminus D \to M$ be a section as above. It follows that
$[C'] - \overline s_*[\partial D] = 0$.
The sign comes from the fact that $\partial D$ is oriented as the boundary
of the disk $D$, which is the opposite to the orientation inherited from
the complement of the disk.
Since the section $\overline s^b_*$ extends
over $D$ and $D$ is null-homotopic, the map $\overline s^b:\partial D \to M$
is homotopic to a constant map $\partial D \to M$.
It follows that $\overline s_* [\partial D] = -b [C]$.
\end{proof}

\begin{rem}
If $\partial \Sigma \neq \emptyset$, then \cref{eq:lem_triv} can be taken
as an alternative definition of the (relative)
Euler number. Indeed, it follows from
the K\"unneth formula that the $[C]$ is not a torsion element
of $H_1(M,\Z)$.
\end{rem}

\begin{definition} \label{def:graph}
A \emph{plumbing graph} is a decorated graph $G$ (with no loops)
with vertex set $\V = \W \amalg \A$,
where each vertex $a\in \A$ has a unique neighbour $w_a$ and $w_a$.
We refer to vertices in $\A$ as \emph{arrowhead vertices}.
$G$ is decorated as follows.
\begin{itemize}
\item
For each $w\in \W$, we have integers $-b_w \in \Z$ and $g_v \in \Z_{\geq 0}$.
These are referred to as the associated \emph{Euler number} (or
sometimes \emph{selfintersection number}) and the \emph{genus}.

\item
Each edge $e$ connecting two vertices in $\W$ is given a sign
$\epsilon_e\in\{+,-\}$.
\end{itemize}

In a drawing of a graph, the genus $g_v$ is written within square
brackets as $[g_v]$ to be distinguished from the Euler number. If it
is omitted, it is assumed to be $0$.
A negative edge will be indicated by the symbol $\circleddash$, whereas
if indication is omitted, the sign is assumed to be positive.
An edge connecting $w\in \W$ and an arrowhead $a\in\A$ is drawn as a dashed
edge, see e.g. \cref{im:Bamboo}. 

Let $M = \cup_{v\in\W} M_i$ be a plumbed manifold and use the notation
introduced in \cref{def:plumbing}.
The \emph{associated
plumbing graph} $G$ has vertex set $\V = \W\amalg\A$ where
$\A = \amalg_{v\in\W} \A_v$, where the elements of $\A_v$ correspond to
the boundary components $B_{v,1}, \ldots, B_{v,e_v}$ of $\Sigma_v$. It has
$e_{v,w}$ edges
connecting $v$ and $w$ if $v,w$ are distinct elements of $\W$
and a single
edge connecting any $a\in\A_w$ with $w$ if $w\in \W$, and no other edges.
Denote by $\E$ this set of edges.

The genus $g_v$ is the genus
of the surface $\Sigma_v$. The Euler number $-b_v$ is the Euler
number of the $S^1$ bundle $M_v\to\Sigma_v$, trivialized on the boundary
components $B_{v,i}$ by the given section, and on the components
$B_{v,w,i}$ by any fiber of $S_{v,w,i} = S_{w,v,i}\to B_{w,v,i}$.

Any edge $e\in\E$ connecting $v,w\in\W$
corresponds to a component $S_{v,w,i} = S_{w,v,i}$
of the intersection $M_v\cap M_w$.
Take fibers $C_v$ and $C_w$ of $\pi_v$ and $\pi_w$, respectively, contained
in $S_{v,w,i}$.
The sign $\epsilon_e$ is defined as the intersection number
of $C_v$ and $C_w$ in $S_{v,w,i}$, that is,
\[
  \epsilon_e = ([C_v],[C_w])_{S_{v,w,i}}.
\]
It follows from definition that
this intersection number is $\pm 1$. This sign depends on the orientation
on $S_{v,w,i}$, which, we recall, is obtained by viewing $S_{v,w,i}$ as
a subset of $\partial M_v$.
\end{definition}

\begin{lemma} \label{lem:edge_signs}
Let $v,w$ be vertices connected by an edge $e$ in a plumbing graph
associated to a plumbed manifold $M$. Let $C_w$ be a fiber
of $\pi_w$ contained in the torus $S_{v,w,i}$ corresponding
to $e$. Then the sign $\epsilon_e$ is positive if and only if
$-C_w$ is an oriented section to the map
$S_{v,w,i}\to B_{v,w,i}$.
\end{lemma}
\begin{proof}
Let $C_v\subset S_{v,w,i}$ be a fiber of $\pi_v$.
We have $\epsilon_e = ([C_v],[C_w])_{S_{v,w,i}}$. Therefore, if $B$ is the
oriented image of some section of $\pi_v|_{S_{v,w,i}}$, then it suffices
to show that $([C_v],[B])_{S_{v,w,i}} = -1$.
By construction, $C_v$ and $B$ intersect in a single point, say
$x \in S_{v,w,i}$, and we can
assume that this intersection is transverse. Let
$c,b \in T_xS_{v,w,i}$ be tangent vectors
inducing positive bases of $T_x C_v$ and $T_x B_v$. Let $a\in T_x M_v$ be
an outward pointing tangent vector. By definition,
$(\pi_v(a), \pi_v(b))$ is a positive basis of $T_{\pi_v(x)} B_{v,w,i}$.
Therefore, $(a,b,c)$ is a positive basis of $T_x M_v$, and so
$(b,c)$ is a positive basis of $T_x S_{v,w,i}$. This means that
$([B],[C_v])_{S_{v,w,i}} = 1$ and so
$([C_v],[B])_{S_{v,w,i}} = -1$.
\end{proof}

\begin{rem} \label{rem:diff}
The above lemma may seem contrary to the usual definition of plumbing
\cite{Neu_plumb,Nem_Szil}. There, the authors start with $S^1$-bundles
over a closed surfaces. The glueing of two pieces, corresponding to an
edge $e$, is made by removing
a tubular neighbourhood around a fiber in each piece and identifying
the boundaries by switching meridians and fibers, multiplied with 
a sign $\epsilon_e$. The output of the two constructions is identical, but
the submanifold $B$ in the proof above, is a meridian, but with the opposite
orientation to that of a standard meridian.
\end{rem}

\begin{example} \label{ex:surf} \cite{Mumford,Nem_Szil}
Let $\X$ be a smooth complex surface and let
$E \subset \X$ be a compact \emph{normal crossing divisor}. This means that
$E$ is a compact reduced
analytic subspace of pure dimension one, decomposing
as $E = \cup_{v\in\V} E_v$ into irreducible components, with the condition
that each $E_v$ is a submanifold of $\X$, that each $E_v$ and $E_w$ intersect
transversally, and that any singularity of $E$ is a double point.
If $T\subset\X$ is a suitable small neighbourhood of $E$, then
$M = \partial T$ is a plumbed manifold, whose plumbing graph $G$ is given
by the intersection matrix of $E$, that is, $G$ has vertex set $\V$,
the genus $g_v$ is the genus of $E_v$, the Euler number $-b_v$ is the
selfintersection number $(E_v,E_v)$, equivalently, it is the Euler number
of the normal bundle of the embedding $E_v\hookrightarrow \X$, and
the number of edges between $v,w\in\V$ is the cardinality $|E_v\cap E_w|$.
Furthermore, $\epsilon_e = +$ for any edge $e$.
\end{example}

\begin{block}
A plumbed manifold $M$ can be recovered from its (decorated) plumbing graph $G$
as follows. As before, denote by $\V$ and $\E$ the set of vertices and edges
in $G$, and by $g_v$ and $-b_v$ the genus and
the selfintersecion number of a vertex $v$ and by $\epsilon_e$ the sign
of an edge. For each $v\in\V$, let $\Sigma_v$ be a compact surface of
genus $g_v$ with $e_v + \sum_{w\in\V\setminus\{v\}} e_{v,w}$ boundary
components, give names $B_{v,w,i}$, for $w\in\V\setminus\{v\}$ and
$1\leq e_{v,w}$ and $B_{v,i}$ for $1\leq i \leq e_v$. Let
$\pi_v:M_v \to \Sigma_v$
be an $S^1$ bundle with sections $s_{v,w,i}$ and $s_{v,i}$ over the boundary
inducing Euler number $-b_v$. The section $s_{v,w,i}$ induces a trivialization
$\phi_{v,w,i}:S^1\times S^1 \to \pi^{-1}_v(B_{v,w,i})$.

We then have
$M \cong \amalg_{v\in\V} M_v / \sim$ where $\sim$ is the equivalence
relation on $\amalg_{v\in\V} M_v$ generated by
$\phi_{v,w,i}(\theta_1,\theta_2)
\sim \phi_{w,v,i}(\theta_2^{-\epsilon_e},\theta_1^{-\epsilon_e})$
where $e$ is the $i^{\mathrm{th}}$ edge connecting $v$ and $w$.
The negative sign in the exponents in the glueing map is
explained by \cref{rem:diff}.

\end{block}



\section{Multiplicities associated with complex valued functions}
\label{s:mult_def}

In this section we give a definition of multiplicities of a
complex valued function on a plumbed manifold under some restrictions
(see \cref{block:mult}). This definition coincides with the multiplicities
associated with fibred links in section 18 of \cite{EisNeu}, if
the function is a fibration over $S^1$.
These multiplicities are useful as they can be obtained by local
computation, but can be used to determine Euler numbers,
see \cref{lem:top_ident}.

\begin{block} \label{block:mult}
Let $M = \cup_{v\in\V} M_v$ be a plumbed manifold with graph $G$, with
vertex set $\V = \W\cup\A$ and let
$\zeta:M\to\C$ be a differentiable function having $0$ as a regular value.
Furthermore, assume that $\zeta$ does not vanish on $\partial M_v$ for all
$v\in\W$.
Thus, $N_v = \zeta^{-1}(0) \cap M_v$ is a closed submanifold
of $M_v$ which does not intersect its boundary.
Assume also that $N_v$ is homologous to a multiple of $[C_v]$ in $M_v$,
that is, $[N_v] = n_v[\pi_v^{-1}(p)]$ for some
(well defined) $n_v\in\Z$.

For any $x\in\Sigma_v \setminus \pi_v(N_v)$, there is a
unique $m_x\in\Z$ so that
$\zeta_*([\pi^{-1}(x)]) = m_x[S^1] \in H_1(\C^*)$. This number is
a locally constant function of $x$. In fact, let
$\xi:[0,1]\to \sigma$ be a $1$-chain connecting $x = \xi(0)$ and
$y = \xi(1)$. We can assume that $\xi$ is an embedding, and
by a small perturbation, we can assume that the map
$N_v\to \Sigma$, induced by $\pi_v$, is an immersion, transverse to $\xi$.
At any intersection point of $\xi$ and $\pi_v(N_v)$, one sees that
$m_{\xi(\cdot)}$ changes by $\pm 1$, depending on the sign of the intersection.
In particular, if $x,y\in \partial \Sigma_v$, then $\xi$ is a cycle
inducing an element $[\xi] \in H_1(\Sigma_v,\partial \Sigma_1,\Z)$. It follows
from the assumptions that we made that
$\pi_{v,*}([N_v]) = 0 \in H_1(\Sigma,\Z)$, and so
\[
  (\pi_{v,*}([N_v]), [\xi])_{\Sigma_v} = 0.
\]
It follows that for $x,y\in \partial \Sigma$, the number $m_x = m_y$ is
a number which well defined by the map $\zeta$; we denote it by $m_v$.
\end{block}

\begin{definition} \label{def:mult}
Let $\zeta:M \to \C$ be as in \cref{block:mult}. We refer to the families
$(m_v)_{v\in\V}$ and $(n_w)_{w\in\W}$ (defined above) as the
\emph{family of multiplicities} and
\emph{dual family of multiplicities} associated with $\zeta$, respectively.
In a drawing of a plumbing graph, a multiplicity is written within
parenthesis, whereas a dual multiplicity is written in parenthesis
next to an arrow emanating from the vertex.
\end{definition}

\begin{lemma} \label{lem:top_ident}
Let $\zeta:M\to\C$ be as in \cref{block:mult}, and let $(m_v)_{v\in\V}$
and $(n_v)_{v\in\W}$ be the associated families of multiplicities and
dual multiplicities. Let $w\in\W$. If $e\in \E_w$ connects $w$ and
$v$, set $m_e = m_v$. We then have
\[
  -b_w m_w + \sum_{e\in\E_w} \epsilon_e m_e = n_v.
\]
\end{lemma}
\begin{proof}
Let $C_w$ be a fiber of $\pi_w$. Since $M_w \cong \Sigma_w\times S^1$, the
element $[C_w]\in H_1(M_w,\Z)$ is nontorsion. It therefore suffices
to show that
\[
  \left(-n_w - b_w m_w + \sum_{e\in\E_w} \epsilon_e m_e\right)[C_w] = 0.
\]
We can assume that $\zeta$ is transversal to the submanifold with
boundary $\R_{\geq 0} \subset\C$ so that $\zeta|_{M_w}^{-1}(\R_{\geq0})$
is a submanifold with boundary $K_w$ in $M_w$. Furthermore, we can assume
that $K_w$ is transversal to $\partial M_w$. This way,
$[\partial K_w] = -[N_w] + \sum_{e\in\E_w} [K_w\cap S_e] \in H_1(M_w,\Z)$.
Let $e\in\E_w$, connecting $w$ and $v\in\V$. Assume that the fiber
$C_w$ was chosen so that $C_w\subset S_e$. Furthermore, let $C_v$ be a fiber
of $\pi_v$ contained in $S_e$ if $v\in\W$, otherwise, let $C_v$ be the
image of $s_v$.
It follows from definition that 
\[
  ([K_e\cap S_e], [C_w])_{S_e} = m_v,\quad
  ([K_e\cap S_e], [C_v])_{S_e} = m_v.
\]
Since $[C_v]$ and $[C_w]$ form a basis of $H_1(S_e,\Z)$, and we have
\[
  ([C_w],[C_w])_{S_e} = ([C_v],[C_v])_{S_e} = 0,\quad
  ([C_w],[C_v])_{S_e} = \epsilon_e,
\]
we get
$[K_e\cap S_e] = \epsilon_e (m_v[C_w] - m_w[C_v])$.
This yields
\[
\begin{split}
  0 = [\partial K_e]
   &= -n_w [C_w] + \sum_{e\in\E_w}\epsilon_e (m_v[C_w] - m_w[C_v]) \\
   &= \left(
        -n_w - b_w m_w + \sum_{e\in\E_w}\epsilon_e m_v
      \right) [C_w].
\end{split}
\]
Here, the variable $v$ inside the sum depends on $e$. The last equality
follows from \cref{lem:triv}
\end{proof}

\begin{example} \label{ex:holom_mult}
Let $\X$ and $E = \cup_{v\in\V} E_v$ be as in \cref{ex:surf}, and let
$h:\X\to\C$ be a holomorphic function. Decompose the divisor
of $h$ as $(h) = (h)_{\mathrm{exc}} + (h)_{\mathrm{str}}$ so that
$(h)_{\mathrm{exc}}$ is supported on $E$, and $(h)_{\mathrm{str}}$ has no
components with nonzero coefficient included in $E$.
We can then write $(h)_{\mathrm{exc}} = \sum_{v\in\V} m_v E_v$, and
$(h)_{\mathrm{str}} = \sum_{D} n_D D$
with $n_D = 0$ if $D = E_v$ for some $v\in\V$.
Assume that the support of $(h)_{\mathrm{str}}$
does not contain any intersection
points in $E$, that is, if $n_D \neq 0$, then $D\cap E_v\cap E_w = \emptyset$
for $v,w\in\W$, $v\neq w$.
If $T\subset \X$ is a small tubular neighbourhood around $E$, then
$M = \partial T$ is a plumbed manifold and $h|_M$ satisfies the
conditions in \cref{block:mult}. The associated family of
multiplicities is $(m_v)_{v\in\W}$.
Furthermore, the family $(n_v)_{v\in\V}$ of dual multiplicities
is given as the intersection $n_v = (E_v, (h)_{\mathrm{str}})$.

Note that here we do not assume
$(h)_{\mathrm{str}}$ to be smooth, only that its intersection points
with $E$ lie in the regular part of $E$.
\end{example}

\section{Negative continued fractions} \label{s:cont_frac}

In this section we discuss negative continued fractions and a plumbing
construction related to them. Some of the notation introduced in this
section follows \cite[III.5]{BHPVdV}.

\begin{block} \label{block:ncf}
Let $a,b$ be relatively prime integers, $b>0$. The fraction $a/b$ can
be written in a unieque way as a (negative) continued fraction
\begin{equation} \label{eq:strictffraction}
  \frac{a}{b} = k_1 - \frac{1}{k_2-\frac{1}{\cdots - \frac{1}{k_s}}}
\end{equation}
where $k_i\geq 2$ for $i\geq 2$. Further, we have $k_1 \geq 2$ if and only if
$a>b$ and $k_1 > 0$ if and only if $a > 0$. 
\end{block}

\begin{definition}
The rational number $a/b$ is called the \emph{(negative) continued fraction}
associated with the sequence $k_1,\ldots, k_s$ and is denoted by
$[k_1, \ldots, k_s]$. The sequence $k_1,\ldots, k_s$ is called the
\emph{(negative) continued fraction expansion} of the rational number
$a/b$.
\end{definition}

\begin{block} \label{block:ncf_extra}
Given $a/b = [k_1, \ldots, k_s]$ as above, 
define integers $\mu_i$ and $\tilde \mu_i$ for $0\leq i \leq s+1$ as follows.
Start by setting
\[
  \mu_0 = 0,\quad \mu_1 = 1, \quad
  \tilde \mu_0 = -1,\quad \tilde\mu_1 = 0.
\]
Then, assuming that we have defined $\mu_j, \tilde\mu_j$
for $0\geq j\geq i$ for some $i>0$, define
\[
  \mu_{i+1} = k_i \mu_i - \mu_{i-1},\quad
  \tilde \mu_{i+1} = k_i \tilde \mu_i - \tilde \mu_{i-1}.
\]
Using induction, one finds
\begin{equation} \label{eq:det_is_one}
  \begin{vmatrix}
    \mu_i       & \mu_{i+1}       \\
    \tilde\mu_i & \tilde\mu_{i+1}       
  \end{vmatrix}
  = 1, \quad
  i=0, \ldots, s.
\end{equation}
Furthermore, the numbers $\mu_i$ and $\tilde \mu_i$ are positive for $i>1$
if $a>0$.
A simple induction on $s$ also proves $\mu_{s+1} = a$ and
$\tilde\mu_{s+1} = b$.
\end{block}

\begin{lemma} \label{lem:ncf_bamboo}
Let $a,b$ be positive integers with no common factors, and let
$k_i, \mu_i, \tilde\mu_i$ be defined as above.
The manifold $M = \bar D\times S^1$ is a plumbed manifold, given as
$M = \cup_{i=1}^s M_i$ where
\[
  M_1 = D_{\frac{1}{s}}\times S^1,\quad
  M_i = (\bar D_{\frac{i+1}{s}} \setminus D_{\frac{i}{s}})\times S^1,\,
  i=2,\ldots,s.
\]
We set $\Sigma_1 = D_{\frac{1}{s}}$ and
$\pi_1:M_1\to\Sigma_1$, $(rt_1, t_2) \mapsto rt_1$ where $r\in\R_{\geq 0}$
and $t_i \in S^1$, as well as
$\Sigma_i = \bar D_{\frac{i}{s}} \setminus D_{\frac{i-1}{s}}$ for $i>1$
and
$\pi_i:M_i\to\Sigma_i$, $(rt_1, t_2)\mapsto r t_1^{\mu_i} t_2^{\tilde\mu_i}$
for $i>1$.
The section over $S^1\subset \Sigma_s$ is given by
$t \mapsto (t^{\tilde\mu_{s+1}}, t^{-\mu_{s+1}})$.
The associated plumbing graph is shown in \cref{im:Bamboo}.
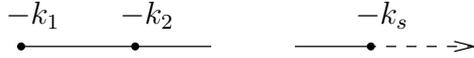
\mynd{Bamboo}{ht}{Plumbing representation of
$S^1\times \overline D$.}{im:Bamboo} 
\end{lemma}
\begin{proof}
It is clear that the given components intersect in tori. Furthermore,
\cref{eq:det_is_one} gives $\gcd(\tilde\mu_i, \mu_i) = 1$. It follows
that $\pi_i$ is an $S^1$ fibration for all $i$.
Another consequence of \cref{eq:det_is_one} is that for
$1\leq i < s$, the map $\pi_i\times\pi_{i+1}:M_i\cap M_{i+1}\to S^1\times S^1$
is a diffeomorphism and that fibers of $\pi_i$ and $\pi_{i+1}$ intersect
positively in the torus $M_i\cap M_{i+1}$.
The same equation shows that the map
$t \mapsto (t^{\tilde\mu_{s+1}}, t^{-\mu_{s+1}})$ is really a section:
\[
  \pi_s(t^{\tilde\mu_{s+1}}, t^{-\mu_{s+1}})
  = t^{\mu_s \tilde\mu_{s+1}-\tilde\mu_s \mu_{s+1}} = t.
\]
Therefore, $M$ is a plumbed manifold with $s$ components.
What is left to show is that the Euler number for the $i^{\mathrm{th}}$
vertex, call it $v_i$, in the graph is $-k_i$.
To see this, consider the function $\zeta:M\to \C$, $(z,t)\mapsto t$. 
The function does not vanish on $M$, and so the dual set of multiplicities
vanish.
We have parametrizations
$S^1\to M_i$, $t\mapsto(r t^{-\tilde \mu_i}, t^{\mu_i})$ of a fiber of
$\pi_i$ for a suitable $r$. Thus, the multiplicities of $\zeta$
are given by $m_{v_i} = \mu_i$, and similarly,
$m_a = \mu_{s+1}$, where
$a$ is the arrowhead. Thus, by \cref{lem:top_ident}, we have
$-b_{v_i} \mu_i + \mu_{i-1} + \mu_{i+1} = 0$ for $1\leq i \leq s$.
Since the same equation holds with $b_{v_i}$ replaced with $k_i$ (and
$\mu_i \neq 0$), we
get $-b_{v_i} = -k_i$.
\end{proof}

\section{Construction} \label{s:constr}

In this section we state our main result in details.
We construct a plumbing graph $G$ from the resolution graph $\Gamma$
along with the multiplicities $m_v$ and $l_v$ of $f$ and $g$.
\Cref{thm:boundary} says that this construction describes the boundary
of the Milnor fiber of the hypersurface singularity given by
$\Phi(x,y,z) = f(x,y) + zg(x,y)$.

\begin{definition} \label{def:constr}
\begin{blist}

\item \label{it:constr_G1}
Let $\Gamma'$ be a connected component of $\Gamma_1$.
Let $\V(\Gamma')$ be the vertex set of $\Gamma'$ and, for $v\in\V(\Gamma')$,
let
$\hat\E_v(\Gamma')$ be the set of edges connecting $v$
and a vertex in $\A_{f,1} \cup \W_2$.
Set also $\hat\E(\Gamma') = \cup_{v\in\V(\Gamma')}\hat\E_v(\Gamma')$.
For any edge $e\in\hat\E(\Gamma')$ connecting $v\in\V(\Gamma')$
and $w\in\A_{f,1}\cup\W_2$, set $v_e = v$ and $w_e = w$ and
$m_e = \gcd(m_v,m_w)$.
For $v\in\V(\Gamma')$, let $\hat\delta_v$ be the number of edges
connecting $v$ and some vertex in $\V(\Gamma')$ or $\hat\V(\Gamma')$.
Let $d_{\Gamma'} = \gcd_{v\in\V(\Gamma') \cup \hat\V(\Gamma')} m_v$ and
define $g_{\Gamma'}, -b_{\Gamma'}$ by the equations
\begin{align}
  d_{\Gamma'}(2-2g_{\Gamma'})
    &=   \sum_{v\in\V(\Gamma')} m_v(2-\hat\delta_v)
	   + \sum_{e\in\hat\E(\Gamma')} m_e, \label{eq:G1_genus} \\
  d_{\Gamma'}(-b_{\Gamma'})
    &=   \sum_{e\in\hat\E(\Gamma')} m_{v_e} l_{w_e} - m_{w_e} l_{v_e}.
       \label{eq:G1_selfint}
\end{align}
Since $d_{\Gamma'} \neq 0$, these are well defined. As we will
see later, we have $g_{\Gamma'},-b_{\Gamma'} \in\Z$.

The graph $G_{\Gamma'}$ has vertex set
$\W(G_{\Gamma'}) = \{ v_{\Gamma',1}, \ldots, v_{\Gamma',d_{\Gamma'}} \}$,
with each vertex decorated by the selfintersection number $-b_{\Gamma'}$
and genus $[g_{\Gamma'}]$.
No two of these vertices are connected by an edge.
Define $G_1$ as the disjoint union of the graphs obtained in this way.

\item \label{it:constr_f1}
Let $a \in \A_{f,1}$ and write
$m_a / m_w = [k_1, \ldots, k_s]$. The graph $G_a$ has $2s+1$ vertices
$v_{a,1,+}, \ldots, v_{a,s,+}, v_{a,1,-}, \ldots, v_{a,s,-}, v_{a,0}$.
There is an edge with sign $\pm$ connecting $v_{a,i,\pm}$ and $v_{a,i,\pm}$
for each $1\leq i \leq s-1$, as well as positive
edges connecting $v_{a,0}$ and $v_{a,s,\pm}$.
All these vertices have genus zero. The vertex $v_{a,i,\pm}$
has selfintersection $-b_{a,i,\pm} = \mp k_i$ and $v_{a,0}$ has
selfintersection number $-b_{a,0} = 0$.

Define $G_{f,1}$ as the disjoint union of these graphs.

\item \label{it:constr_bridge}
Let $v_1\in\W_1$ and $v_2\in\W_2$ be vertices of $\Gamma$ connected by
an edge $e$ and write $m_{v_2}/m_{v_1} = [k_1,\ldots,k_s]$.
The graph $G_e$ has vertices
$v_{e,1,+}, \ldots, v_{e,s,+}$, $v_{e,1,-}, \ldots, v_{e,s,-}, v_{e,0}$,
each with genus zero. The vertex $v_{e,i,\pm}$ has the selfintersection
number $-b_{e,i,\pm} = \mp k_i$ and $v_{e,0}$ has selfintersection
$b_{e,0} = 0$.
We have an edge with sign $\pm$ connecting $v_{e,i,\pm}$ and
$v_{e,i+1,\pm}$, as well as positive edges connecting
$v_{e,s,\pm}$ and $v_{e,0}$.

Let $G_{\mathrm{b}}$ be the disjoint union of graphs obtained in this way.

\item \label{it:constr_G2}
The graph $G_2$ is defined as follows. For each $w\in\W_2$, we have
two vertices $v_{w,+}, v_{w,-}$ in $G_2$ and these are all the
nonarrowhead vertices of $G_2$. They are decorated by genus zero and have
selfintersection number $-b_{w,\pm} = \mp b_w$, where
$-b_w$ is the selfintersection number of $w$ in $\Gamma$.

\item \label{it:constr_g2}
Let $a\in\A_{g,2}$. The graph $G_a$ has nonarrowhead vertices
$v_{a,0}, \ldots, v_{a,m_w}$ where $w = w_a$, each of genus zero.
The vertex $v_{a,0}$ has selfintersection number $-b_{a,0} = 0$, whereas
$v_{a,i}$ has selfintersection $-b_{a,i} = -l_a$.
For each $1\leq i \leq m_w$ there is a negative edge connecting
$v_{a,0}$ and $v_{a,i}$.

Let $G_{g,2}$ be the disjoint union of graphs obtained in this way.

\end{blist}
\end{definition}

\begin{definition} \label{def:bridges}
The graph $G$ is the disjoint union of the graphs
$G_1$, $G_{f,1}$, $G_{\mathrm{b}}$, $G_2$, $G_{g,2}$,
with the following additional edges.

\begin{blist}

\item \label{it:bridges_f1}
For $\Gamma'\subset\Gamma_1$ a connected component,
$1\leq i \leq d_{\Gamma'}$ and $a\in\A_{f,1}$, connect $v_{\Gamma',i}$
and $v_{a,0}$ with $m_e/d_{\Gamma'}$ negative edges, where
$m_e$ is as in \cref{def:constr}\cref{it:constr_G1}.

\item \label{it:bridges_br1}
Similarly, assuming that $\Gamma'\subset\Gamma_1$ is a connected component,
$1\leq i \leq d_{\Gamma'}$ and that $v_1\in\W(\Gamma')$ and $v_2\in\W_2$
are connected by an edge $e\in\E(\Gamma)$. Connect
$v_{\Gamma',i}$ and $v_{e,0}$ with $m_e/d_{\Gamma'}$ negative edges
and connect $v_{w,\pm}$ and $v_{a,1,\pm}$ with an edge with sign $\pm$.

\item \label{it:bridges_g2}
Let $a\in\A_{g,2}$ and $w = w_a$. The vertex $v_{a,0}$ is connected to
both $v_{w,+}$ and $v_{w,-}$ by a positive edge.

\end{blist}

\end{definition}

\begin{thm} \label{thm:boundary}
The boundary of the Milnor fiber of the singularity $f(x,y) + zg(x,y) = 0$
at the origin is a plumbed manifold with plumbing graph $G$.
\end{thm}

\begin{rem} \label{rem:blow_down_g2}
\begin{blist}

\item
Let $a\in\A_{g,2}$ and assume that $l_a = 1$. In this case, the vertices
$v_{a,1},\ldots,v_{a,m}$, where $w = w_a$ and $m = m_w$, blow down to
simplify the graph (see \cite{Neu_plumb} for blowing down). This operation
removes these vertices, and replaces the Euler number $-b_{a,0} = 0$
with $-b_{a,0} = m$.

\item
We can apply the operation R0(a) from \cite{Neu_plumb} to the vertices
$v_{\Gamma',i}$ for a connected component $\Gamma'\subset \Gamma_1$
as well as to $v_{a,i}$ for $a\in\A_{g,2}$ and $1\leq i \leq m_{w_a}$.
This way, all the edges adjacent to these vertices will be positive
instead of negative. Note, however, that this also changes the sign of
the corresponding multiplicities given in \cref{s:mult}.

\end{blist}
\end{rem}

\section{A multiplicity system for $z$} \label{s:mult}

In this section, we give multiplicities and dual multiplicities
for the function $z$. For simplicity, the multiplicitiy and
dual multiplicity for a vertex $v_*$ constructed in
\cref{def:constr} will be denoted by $m_*$ and $n_*$.
The proof of \cref{thm:mult} is given in \cref{s:proofs}.

\begin{thm} \label{thm:mult}
The restriction $z|_M$ of the coordinate function $z$ to the boundary
$M = \partial F_{\Phi}$ of the Milnor fiber of $\Phi$ satisfies the
conditions given in \cref{block:mult}. Furthermore, the associated
families of multiplicities $(m_v)_{v\in\V(G)}$ and
dual multiplicities $(n_v)_{v\in\W(G)}$
are is given as follows. 

\begin{blist}

\item \label{it:mult_G1}
If $\Gamma'\subset\Gamma_1$ is a connected component, then
$m_v = 1$ and $n_v = 0$ for $v = v_{\Gamma',i}$, $i=1, \ldots, d_{\Gamma'}$.

\item \label{it:mult_f1}
Let $a\in\A_{1,f}$ be an arrowhead connected to $w\in\W_1$ in $\Gamma$ and
set $\tilde m_w = m_w / \gcd(m_w,m_a)$ and $\tilde m_a = m_a / \gcd(m_w,m_a)$.
Write $m_w/m_a = \tilde m_w/\tilde m_a = [k_1, \ldots, k_s]$ and define
$\mu_i, \tilde\mu_i$
as in \cref{block:ncf}. The multiplicities of $z$ are given by
\[
  m_{a,i,+} = (m_w - l_w) \mu_i - m_a \tilde\mu_i,\quad
  m_{a,i,-} = l_w \mu_i
\]
for $i=1,\ldots, s$ and $m_{a,0} = -\tilde m_a l_w$.
The dual multiplicities for these
vertices are given by $n_{a,1,+} = m_a$, and $0$ otherwise.

\item \label{it:mult_bridge}
Let $v_1, v_2$ and $e$ be as in \cref{def:constr}\cref{it:constr_bridge}.
Let $\tilde m_i = m_{v_i} / \gcd(m_{v_1},m_{v_2})$ for $i=1,2$.
Write $m_{v_2}/m_{v_1} = \tilde m_2 / \tilde m_1 = [k_1,\ldots,k_s]$
and define $\mu_i, \tilde\mu_i$ for $i=1,\ldots,s$
as in \cref{block:ncf}. Then 
$m_{e,0} = \tilde m_1 l_{v_2} - \tilde m_2 l_{v_1}$ and
\[
  m_{e,i,+} = (m_{v_1} - l_{v_1}) \mu_i
               - (m_{v_2} - l_{v_2}) \tilde \mu_i, \quad
  m_{e,i,-} = l_{v_1} \mu_i - l_{v_2} \tilde \mu_i
\]
The dual multiplicities vanish on these vertices.

\item \label{it:mult_G2}
Let $w\in \W_2$. Then $m_{w,+} = m_w - l_w$ and $m_{w,-} = l_w$.
The dual multiplicities are given by $n_{w,+} = \sum_{w_a = w} m_a$
and $n_{w,-} = 0$.

\item \label{it:mult_g2}
Let $a\in \A_{g,2}$ and set $w = w_a$.
Then $m_{a,0} = -l_a$ and $m_{a,i} = 1$ for $i=1,\ldots,m_w$.
The dual multiplicities associated with $v_{a,i}$ vanish.

\end{blist}
\end{thm}

\begin{rem}
Let $e,v_1$ be as in \cref{thm:mult}\cref{it:mult_bridge}. One proves
that, in fact, $m_{e,1,+} = m_{v_1} - l_{v_1}$ and $m_{e,1,-} = l_{v_1}$.
\end{rem}

\section{Examples} \label{s:examples}

\begin{example}
The singularity $T_{a,b,\infty}$ is the singularity at the origin
of the hypersurface given by $x^a + y^b + xyz = 0$.
In the case $b=2$, the boundary of the Milnor fiber has been
described in \cite{Nem_Szil}.
We take $f(x,y) = x^a + y^b$ and $g(x,y) = xy$.
We will assume $a,b$ satisfying $a\geq b \geq 2$ and $a>2$.
We claim
that the boundary of the Milnor fiber of this singularity is given by
the plumbing graph

\mynd{Tabs}{ht}
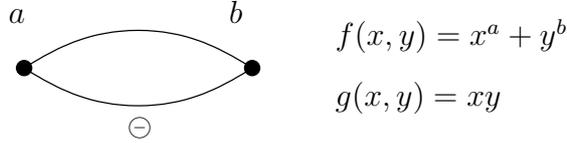
{A plumbing graph for the boundary of the Milnor fiber of the singularity
$T_{a,b,*}$, $x^a + y^b + xyz$.}
{im:Tabs} 

Let $\phi:V\to \C^2$ be the minimal resolution of the plane curve
$fg$ and let $\Gamma$ be its resolution graph.
Then $\Gamma$ is a 
string with two arrowheads corresponding to $g$, one on each end
of the string, as well as $d := \gcd(a,b)$ arrowheads corresponding to $f$.
Name the nonarrowhead vertices
of the graph $v_1, \ldots, v_s$ so that $v_i, v_{i+1}$
are adjacent. Let $-b_i$ be the selfintersection number associated with the
vertex $v_i$. There is a unique $j$ so that $-b_j = -1$.
The set $\A_f$ consists of $d$ arrowheads, each connected
to $v_j$, whereas $\A_g$ consists of two arrowheads, one connected to
$v_1$ and the other to $v_s$.
Write also $m_i, l_i$ for the multiplicities of $f$ and $g$
on $v_i$.

\emph{Claim:} We have $m_1 \geq l_1$ and
$m_s > l_s$ and $m_i > l_i$ for $i=2,\ldots,s-1$.

In fact, using \cite[Lemma 20.2]{EisNeu}, one finds $m_j = ab/d$ and
$l_j = a/d + b/d$. It follows from our assumptions that
$m_j - l_j > 0$. Now, define integers $r_i = m_i - l_i$ for $i=1,\ldots, s$
and $r_0 = r_{s+1} = -1$. We then have
\[
  r_{i-1} - b_i r_i + r_{i+1} = 0, \quad
  i = 1,\ldots, \hat j, \ldots, s.
\]
It follows easily that this sequence increases strictly from $r_0 = -1$
to $r_j = m_j-l_j$, and then decreases strictly from $r_j$ to $r_{s+1} = -1$.
Since these are integers, the claim follows.

We leave to the reader to show that the equality $m_i = l_i$
holds for $i=1$ or $i=s$ if and
only if $b=2$, the case already covered by
N\'emethi and Szil\'ard \cite{Nem_Szil}. This can be achieved by calculating
$m_i$ and $l_i$ explicitly using Lemma 20.2 of \cite{EisNeu}.

We start by showing how the above graph is obtained from
the output of the algorithm in the case when
$l_i > m_i$ for all $i$. Since $\W = \W_2$, the graph $G_2$ consists
of two strings, one of them identical to $\Gamma$, the other one having
Euler numbers with opposite signs and negative edges. In addition, we
have $\A_{g,2} = \{a_x, a_y\}$, two arrowhead vertices corresponding to
the strict transform of the factors $x$ and $y$ of $g$. As described
in \cref{rem:blow_down_g2}, the graph $\Gamma$ can be taken as these
two strings, connected on each end by vertices with Euler number
$m_x$ and $m_y$. These are the multiplicities of $f$ along the components
on the end of the string. It follows from \cite{EisNeu} that these
multiplicities are $a$ and $b$. Furthermore, the two strings blow down
(we can blow down the vertices one by one in the opposite order in
which they appear during the process of resolving $f$). Each string is
replaced by an edge, the first string by a positive edge, the second
one by a negative edge. Below, we explicate the case when $a = 7$
and $b=5$.

\mynd{5_7_res}{ht}
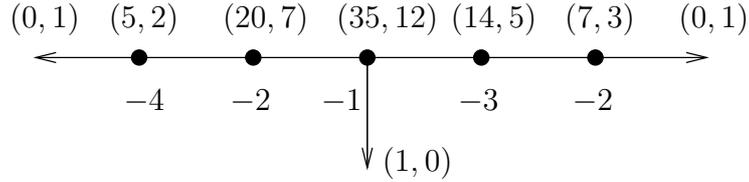
{A resolution graph of the plane curves $f(x,y) = x^7+y^5=0$
and $g(x,y) = xy$, along with their multiplicities.}
{im:5_7_res} 

\mynd{5_7_output}{ht}
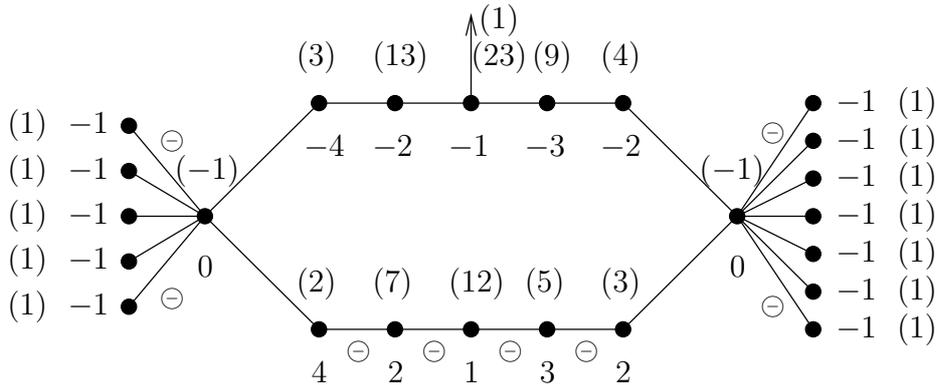
{Output of the algorithm for $\Phi(x,y,z) = x^7 + y^5 + xyz$.}
{im:5_7_output} 

In the case when either $m_1 = l_1$ or $m_s = l_s$, the algorithm has,
in fact, the same output. We let it suffice to clarify this principle by
considering an example.
Take $a = 3$ and $b = 2$. A resolution graph $\Gamma$, decorated
with the pairs of multiplicities $(m_v,l_v)$ is shown in \cref{im:3_2_res}.

\mynd{3_2_res}{ht}
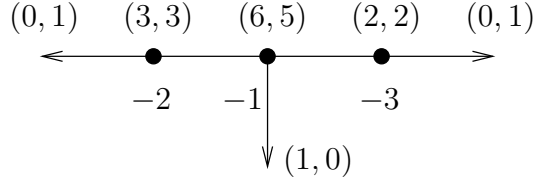
{A resolution graph of the plane curves $f(x,y) = x^3+y^2=0$ and
$g(x,y) = xy$, along with their multiplicities.}
{im:3_2_res} 
We see that $\W_2$ now only contains the vertex $v_2$, whereas
$v_1,v_3 \in \W_1$, each providing a connected component of $\Gamma_1$.
We order the vertices in \cref{im:3_2_res} in such a way that $b_1 = -2$
and $b_3 = -3$. Applying \cref{def:constr}\cref{it:constr_G1} to the
component $\Gamma'$ of $\Gamma_1$, containing only the vertex $v_1$,
as well as $\Gamma''$ containing only $v_3$,
we get
\[
  d_{\Gamma'} = 3, \quad
  g_{\Gamma'} = 0, \quad
  -b_{\Gamma'} = -1, \quad
  d_{\Gamma''} = 2, \quad
  g_{\Gamma''} = 0, \quad
  -b_{\Gamma''} = -1.
\]
We get five new vertices. The edges $e_1 = \{v_1,v_2\}$ and
$e_2 = \{v_2,v_3\}$ are
of the form described in \cref{def:constr}\cref{it:constr_bridge}. The five
new vertices are connected to $v_{e_1,0}$ and $v_{e_2,0}$ to obtain the
graph in \cref{im:3_2_output}, which also shows the multiplicities of
the function $z$.
After blowing down, we obtain \cref{im:Tabs}.
\mynd{3_2_output}{ht}
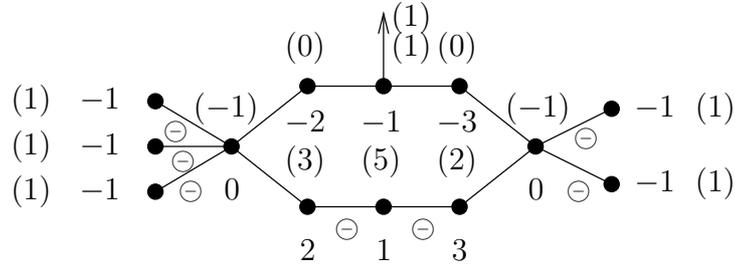
{Output of the algorithm for $\Phi(x,y,z) = x^3 + y^2 + xyz$.}
{im:3_2_output} 

\end{example}

\begin{example}
Consider the plane curves
\[
\begin{split}
  f(x,y) &= (x^2 + \lambda_1 y^3)  ( (x^2+y^3)^2 + \mu_1 x^5)^2, \\
  g(x,y) &= (x^2 + \lambda_2 y^3)^3( (x^2+y^3)^2 + \mu_2 x^5),
\end{split}
\]
where $\lambda_1,\lambda_2 \in \C\setminus\{0,1\}$ are distinct and
$\mu_1,\mu_2 \in \C\setminus\{0\}$ are distinct.
The resolution graph $\Gamma$, decorated with the multiplicities $m$ and $l$
is given in \cref{im:omni_res}. The set $\W_1$ consists of the first three
vertices appearing during the resolution process, corresponding to the
first Puiseux pair, where $f$ and $g$ have equal multiplicities. But, as
$f$ has more components with two Puiseux pairs, compared with $g$,
the multiplicities of $f$ are higher along the second part of the resolution
process, that is, the three vertices appearing last.

\mynd{omni_res}{ht}
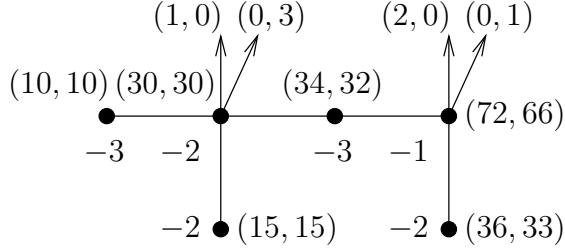
{Plane curves whose branches have $1$ and $2$ Puiseux pairs.}
{im:omni_res} 

The graph $\Gamma_1$ is connected, and the invariants from
\cref{def:constr}\cref{it:constr_G1} are easily computed using
\cref{eq:G1_genus} and \cref{eq:G1_selfint}:
\[
\begin{array}{rclrcl}
  d_{\Gamma_1} &=& \gcd\{ 10, 30, 15, 1, 34 \}, &
  d_{\Gamma_1} &=& 1, \\
  1\cdot(2-2g_{\Gamma_1}) &=& 1\cdot 15 + 1\cdot 10 + (-2)\cdot 30 + 1+2,&
  g_{\Gamma_1} &=& 17,\\
  1\cdot(-b_{\Gamma_1}) &=& (30\cdot 0 - 1\cdot 30) + (30\cdot 32-30\cdot 34),&
  -b_{\Gamma_1} &=& -90.
\end{array}
\]
Thus, we obtain a single vertex $v_{\Gamma_1}$ with genus $17$ and
Euler number $-90$. Furthermore, $m_{\Gamma_1} = 1$.

The set $\A_{f,1}$ contains one element. Using the notation in
\cref{def:constr}\cref{it:constr_f1}, we have $m_a = 1$, $m_w = 30$
and $l_w = 30$. 
We have $m_a/m_w = [1,2,\ldots,2]$ where the number of $2$'s is $29$.
Therefore, $61$ new vertices are created. The vertex
$v_{a,0}$ is connected to $v_{\Gamma_1}$ by $m_e = 1$ edge with a
negative sign. Furthermore,
\[
  m_{a,0} = -30, \quad m_{a,1,+} = 0,\quad
  m_{a,1,-} = 30, \quad n_{a,1,+} = 1.
\]
Note that the vertices $v_{a,i,\pm}$ blow down, leaving only
the vertex $v_{a,0}$.

There is one edge $e$ connecting $\W_1$ and $\W_2$. In the notation of
\cref{def:constr}\cref{it:constr_bridge}, we have $m_e = 2$ and
\[
\begin{array}{lcrlcrlcr}
  m_{v_1}     &=& 30,\quad\quad
  \tilde m_1  &=& 15,\quad\quad
  l_{v_1}     &=& 30, \\
  m_{v_2}     &=& 34,\quad\quad
  \tilde m_2  &=& 17,\quad\quad
  l_{v_2}     &=& 32,
\end{array}
\]
We have $17/15 = [2, 2, 2, 2, 2, 2, 2, 3]$, hence the creation of
$17$ new vertices. The vertices $v_{e,0}$ and $v_{\Gamma_1}$ are connected
by two edges with negative sign. Furthermore, we get
\[
  m_{e,0} = -30, \quad
  m_{e,1,+} = 0, \quad
  m_{e,1,-} = 30.
\]

The set $\W_2$ contains three elements, inducing six new vertices with
genus $0$ and Euler number $\mp 3, \mp 1, \mp 2$. The $\mp3$ curves are
connected with the vertices $v_{e,1,\pm}$ from the previous construction.

The set $\A_{g,2}$ contains one vertex, say $a \in \A_{g,2}$,
as in \cref{def:constr}\cref{it:constr_g2}. Set also $w = w_a$.
Since $m_w = 72$ and $l_a = 1$, we get a vertex $v_{a,0}$ with genus
$0$ and Euler number $0$, connected to the vertices $v_{a,1},\ldots,v_{a,72}$
with genus $0$ and Euler number $-1$, via a negative edge.
We have $m_{a,0} = -1$ and $m_{a,i} = (1)$ for $1\leq i \leq 72$.

\mynd{omni_output}{ht}
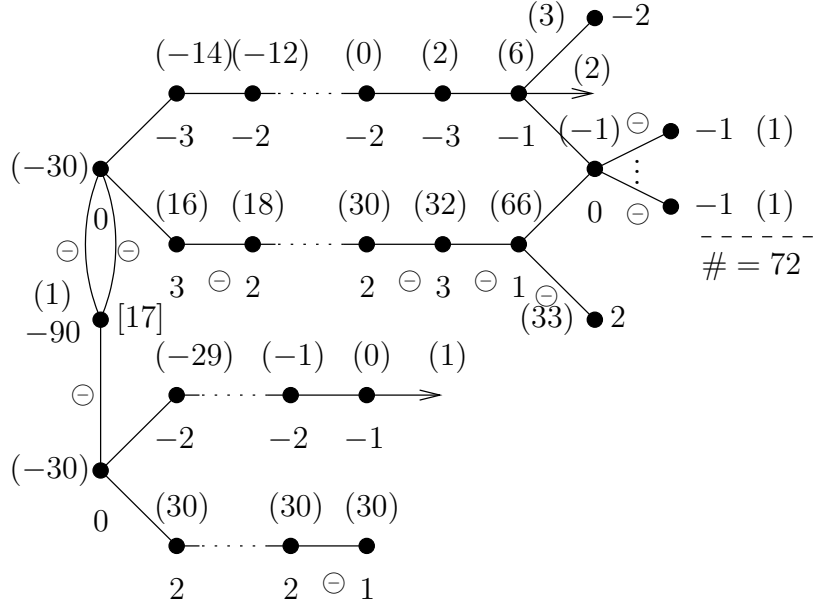
{The dotted lines denote either a string of $(-2)$-pieces connected
by positive edges or $2$-curves connected by negative edges.}
{im:omni_output} 

\end{example}

\section{Proofs} \label{s:proofs}

To prove \cref{thm:boundary,thm:mult}, we start by defining pieces
$M_v \subset M$ and projections $\pi_v:M_v \to \Sigma_v$ for all 
vertices $v\in \V$, using the description in \cref{thm:whole_fiber}.
From the construction, it will be clear that $M = \cup_{v\in \V} M_v$, and
that individual pieces intersect according to the edges of $G$. 
Finally, we verify the formulas for genera and
selfintersection numbers. In fact, it will be clear that the genus decoration
is zero, except for in the case of $v_{\Gamma'}$, where an argument
similar to A'Campo's formula \cite{ACampo} is used.
Similarly as in \cite{Nem_cyclic,Nem_Szil}, nontrivial
Euler numbers are determined using the multiplicities of
$z$ and \cref{lem:top_ident}. We note that the proof of \cref{thm:mult},
can be carried out as soon as the projections $\pi_v$ are defined. In
particular, this proof does not use the Euler numbers, which are computed
using the multiplicities of $z$.

\begin{proof}[Proof of \cref{thm:boundary}]
We start by providing sets $M_v$ for each vertex of the graph $G$. We
then prove that these
pieces provide a plumbing structure on $M$ with the plumbing graph $G$.
\begin{blist}

\item \label{it:pf_G1}
Let $X'_1$ be the closure of the set
\[
  \overline T_{f,g} \cap T_1 \setminus (T_{f,1} \cup T_2).
\]
By construction, this is a closed tubular neighbourhood of
\[
  F_{f,1} := F_f \cap T_1 \setminus (T_{f,1} \cup T_2),
\]
in particular, we have a disk bundle $X'_1 \to F_{f,1}$. We can assume
that the intersection of this disk bundle with the divisor associated
with $g$ is a set of disks. Let $X_1$ be the four manifold
obtained from $X'_1$ by twisting along these disks as in
\cref{def:T_f_g} and let $M_1 \subset \partial X_1$ be the
associated $S^1$ bundle. It is then clear that we have
$M_1 \subset M_{f,g}$, that the boundary of $M_1$ consists of tori
and that $M_1$ is in a natural way an $S^1$ bundle over $F_{f,1}$.

Let $\Gamma'\subset \Gamma_1$ be a component as in
\cref{def:constr}\cref{it:constr_G1}. Setting
\[
  X'_{\Gamma'} = X'_1 \cap \bigcup_{v\in\V(\Gamma')} T_v,
\]
we obtain correspondingly $X_{\Gamma'} \subset X_1$ and
$M_{\Gamma'} \subset M_1$. This way, $M_{\Gamma'}$ is an $S^1$ bundle
over the surface $F_{f,\Gamma'} = F_f \cap X'_{\Gamma'}$.

Firstly, we note that the number of connected components of $F_{f,\Gamma'}$ is
precisely $d_{\Gamma'}$ and that, furthermore, the monodromy permutes these
components cyclically.
This follows from Proposition 2.20 of \cite{Nem_cyclic}, see also 2.21 of
the same article.

Secondly, the genus of the components of $F_{f,\Gamma'}$ is
$g_{\Gamma'}$, satisfying \cref{eq:G1_genus}. This follows from
a small generalization of
A'Campo's formula \cite{ACampo} which gives
\[
  \chi(F_{f,\Gamma'}) = \sum_{v\in\V(\Gamma')} m_v (2-\hat\delta_v).
\]
What is more, $F_{f,\Gamma'}$ has $m_e$ boundary components close to
the intersection of $E_v$ and $E_{w_e}$ for $e\in \hat\E(\Gamma')$.
Thus, $F_{f,\Gamma'}$ has a total of
$\sum_{e\in\hat\E(\Gamma')} m_e$ boundary components.

The formula \cref{eq:G1_selfint} is verified below.

\item \label{it:pf_f1}

Let $a\in A_{1,f}$ and set $w = w_a$. Define
$M_a = \partial \overline T_{f,g} \cap \overline T_a$
We have coordinates $u,v$ on $T_a$ so that
$T_a \cong \set{(u,v)}{|u| \leq 1, \, |v| \leq 1}$ and so that
$E_a \cap T_a$ and $E_w \cap T_a$ are the vanishing sets of $u$ and $v$,
respectively.
We can then write $M_a = M_{a,+} \cup M_{a,-} \cup M_{a,0}$ where for certain
$0 < \eta \ll \epsilon \ll 1$, we have
\[
  M_{a,+} = \set{(u,v)}{|v| = 1     ,\,|u| \leq 1},\quad
  M_{a,-} = \set{(u,v)}{|v| = \eta  ,\,|u| \leq 1}
\]
and $M_{a,0}$ is defined by setting
\[
  M'_{a,0} = \set{(u,v)}{|u| = 1,\, \eta \leq |v| \leq 1},\quad
  M_{a,0} = M'_{a,0} \setminus N,
\]
where $N$ is an $\epsilon$ neighbourhood around $F_f \cap M_{a,0}'$.
The projection of this picture via $(u,v) \mapsto (|u|,|v|)$ is shown
in \cref{im:scheme_f1}.
\mynd{scheme_f1}{ht}
{A diagram showing what happens near $E_a \cap E_{w_a}$.}
{im:scheme_f1} 
We can assume that in the coordinates $u,v$, we can write
$f|_{T_a}(u,v) = u^{m_a} v^{m_w}$. Define
$\tilde m_a = m_a / m_e$ and $\tilde m_w = m_w / m_e$.
We find
\[
  F_f \cap M'_{a,0}
    = \coprod_{j=0}^{m_e-1}
      \set{(e^{(-t+ j/m_a)\tilde m_w 2\pi i}, e^{t \tilde m_a 2\pi i})}
          {t \in [0,1]}
\]
In fact, we have an $S^1$ bundle projection
$\pi'_{a,0}$ mapping $M'_{a,0}$ to an annulus by the formula
$\pi'_{a,0}(u,v) = u^{\tilde m_a} v^{\tilde m_v}$. This way,
$F_f \cup {M'}_a^0$ consists of $m_e$ fibers of $\pi'_{0,a}$.
In particular, we can assume that $\pi'_{a,0}$ restricts to an
$S^1$ bundle $\pi_{0,a} = \pi'_{0,a}|_{M_{0,a}}$.
We orient the fibers so that one of them is parametrized as
$t \mapsto (t^{-\tilde m_w}, t^{\tilde m_a})$, which induces an orientation
on the target space of $\pi_{0,a}$.

By \cref{lem:ncf_bamboo}, the manifold $M_{a,+}$ can be given as a plumbed
manifold with plumbing graph as in \cref{im:Bamboo}, where
$[k_1,\ldots,k_s] = m_a / m_w$, so that the section corresponding to the
arrowhead to the right can be chosen to coincide with a fiber of $M_{a,0}$,
with the opposite orientation.
Furthermore, we have an orientation reversing
diffeomorphism $M_{a,+} \to M_{a,-}$ given by
$(u,v) \mapsto (u,\eta v)$. This way, we see $M_{a,-}$ as a plumbed
manifold with the same plumbing graph, modified by changing signs on
all selfintersection numbers as well as edges.

At this point, we have shown that the manifold $M_a$ is a plumbed manifold
with plumbing graph $G_a$, with $m_e$ dashed arrows added to $v_{a,0}$,
except we did not specify a section corresponding to these arrowheads.
Furthermore, (and this cannot be done without the sections) we have not
determined the Euler number associated with $v_{a,0}$.
Since $M_{0,a}$ is obtained by removing a tubular
neighbourhood around a fiber of the projection $\pi'_{0,a}$, we can choose
as a section a meridian around this fiber.
Note that this section is exactly
a fiber of the projection $\pi_{\Gamma',i}$ for a suitable
$1\leq i \leq d_{\Gamma'}$, where $w$ is a vertex of the component $\Gamma'$
of $\Gamma_1$. But we can be more specific. Let $\psi:S^1 \to M_{a,+}$
be a parametrization of a fiber in the boundary component of $M_{a,s,+}$.
This induces a map $S_1\times [\eta,1] \to {M'}_{a,0}$ which is a global section
to the fibration of $M'_{a,0}$, restricting to a global section to
the fibration of $M_{a,0}$, which again restricts to a parametrization of
the fibers of $M_{a,s,\pm}$, as well as a parametrization of a meridian around
$F_f \cap M_{a,0}$. This shows that with this choice of sections on the
boundary, the Euler number of the bundle $M_{a,0}$ is $0$.

Finally, we note that $M_{a,0}$ intersects $M_{\Gamma'}$ in exactly
$m_e$ tori, and that the number of these tori in each component of
$M_{\Gamma'}$ is the same. It follows from the construction that
in each of these tori, a fiber of $M_{a,0}$ and a fiber from $M_{\Gamma'}$
form an integral basis on homology. Furthermore, one verifies
that an oriented fiber of $M_{a,0}$ in such a torus is an oriented
section of $M_{\Gamma'}$. This can be seen by noting that both
wind around $E_a$ with multiplicity $-\tilde m_a$.
Therefore, these tori yield edges
with a negative sign, by \cref{lem:edge_signs}. These are the
edges defined in \cref{def:bridges}\cref{it:bridges_f1}.

\item \label{it:pf_bridge}
Let $e$ and $v_i \in \W_i$ be as in \ref{def:constr}\cref{it:constr_bridge}.
Let $D$ be a disc in $E_{v_1}$ with center the intersection point
of $E_{v_1}$ and $E_{v_2}$ corresponding to $e$
which is a slightly bigger than the corresponding disk in
$E_{v_1} \cap \overline T_{v_2}$. We can add the preimage of $D$
in $\overline T_{v_1} \setminus T'_{v_2}$ to $\overline T_{f,g}$ without
changing its diffeomorphism type. From here, the proof follows similarly
as in the previous case. A schematic picture is shown in
\cref{im:scheme_bridge}.
\mynd{scheme_bridge}{hbt}
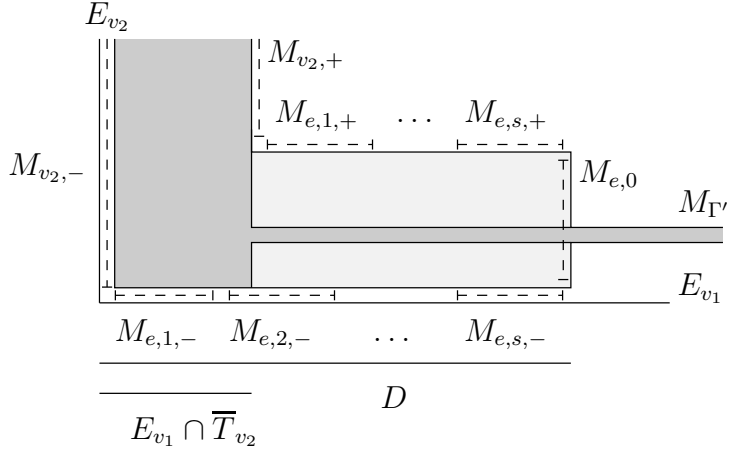
{A diagram showing what happens near $E_{v_1} \cap E_{v_2}$.}
{im:scheme_bridge} 
\item \label{it:pf_G2}
Let $w \in \W_2$ and define $M_w$ as the closeure of
$\overline T_w \setminus \cup_v \overline T_v$, where the union $\cup_v$
ranges over $v\in \W \cup \A_g\setminus\{w\}$. It follows from construction that
$M_w$ consists of two copies of an $S^1$ bundle over the surface $E_w$ with a
disk removed for each neighbour in $\W \cup \A_g\setminus\{w\}$.
Indeed, these are
the corresponding subsets of the boundaries of $T_w$ and $T'_w$ (recall
\cref{def:neighbourhoods}). Write $M_{w,+}$ and $M_{w,-}$ for the outer
and inner components. These components will correspond to the vertices
$v_{w,\pm}$.

These fibrations extends canonically over
the disks removed from $E_w$
by taking $T_w$ and $T'_w$, and we can take a meridians around
central fibers as the trivializing section on the boundary.
It follows immediately that the two $S^1$ bundles have relative Euler
numbers $-b_{w,\pm} = \mp b_w$. Furthermore, since $E_w$ is a rational
curve, the two vertices $v_{w,\pm}$ have associated genus $0$.

As both components of $M_w$ are boundaries of similar tubular neighbourhoods,
they can be identified, but the inner one, i.e. the boundary of
$T'_w$ has its orientation reversed. We will consider this part as a fibration
over the same base as the outer component. Therefore, a fiber in
$M_{w,+}$ is a meridian around $E_w$, whereas a fiber in $M_{w,-}$ is a
(relatively small) meridian around $E_w$ with the orientation reversed.

If $w,w'\in\W_2$ are joined by an edge, it follows easily that the two
components corresponding to $w$ intersect with those of $w'$
in the same way as prescribed by the resolution graph $\Gamma$.

\item \label{it:pf_g2}
Finally, we describe what happens close to a component of the strict
transform of $g$ corresponding to $a\in\A_{g,2}$.

Let $M'_a = M' \cap \overline T_a$ and $M_a$ the corresponding twisted
subset of $M$. Let $M_{a,0} = M \cap \partial T_a$. It follows from
construction that $M_a$ fibers by a map $\pi_{a,0}$
over the disk $E_a$ with $m_a + 1$
smaller disks removed, one corresponding to $T'_{w_a}$, and
$m_a$ of them corresponding to $T_\epsilon$.
We orient the fiber to coincide with that of a meridian around $E_a$.
This chooses an orientation of the base space
of $\pi_{a,0}$, the
opposite of the standard one on $E_a$.
This bundle is
trivialized in a similar way as in \cref{it:pf_f1},
yielding Euler number $-b_{a,0} = 0$.
It is also clear that $g_{a,0} = 0$.

The closure of $M_a \setminus M_{a,0}$ is an $S^1$ bundle over
$F_f \cap T_a \cong \amalg_{m_w} \overline D$. This gives $m_w$ pieces
$M_{a,1},\ldots, M_{a,m_w}$, ordered arbitrarily.
The fibers are meridians around $F_f$.

We see that $M_a$ is a plumbed manifold with plumbing graph $G_a$ (with
some dashed arrows added, corresponding to the boundary).
Using \cref{lem:edge_signs}, we see that the edges between
$v_{a,0}$ and $v_{a,i}$ have a negative sign, whereas the edges
connecting $v_{a,0}$ and $v_{w,\pm}$ are positive.
\end{blist}

In \crefrange{it:pf_G1}{it:pf_g2} above we have assigned subsets
$M_v \subset M$ to each vertex $v$ of the graph $G$ constructed in
\cref{def:constr,def:bridges}.
It is clear that each piece is connected and that each boundary components
of any of the pieces are tori. Furthermore, the components of intersection
of two pieces correspond to the edges connecting the corresponding vertices.
The base space of each fibration is a surface
of the genus specified, or zero otherwise.

The only part which remains to prove is \cref{eq:G1_selfint}. 
But this follows immediately from \cref{lem:top_ident} and
\cref{thm:mult}.
\end{proof}

\begin{proof}[Proof of \cref{thm:mult}]

\begin{blist}

\item \label{it:pfm_G1}
Let $\Gamma'$ be as in \cref{def:constr}\cref{it:constr_G1}.
It follows from the proof in \cite{Sigurdsson2016} that $|z|$ is 
constant on $M_{\Gamma',i}$ (and nonzero). It follows that
the dual multiplicities $n_{\Gamma',i}$ vanish.
A fiber in $M_{\Gamma',i}$ is an oriented meridian around $F_f$.
The restriction of $z = (f-\epsilon)/g$ to such
a fiber is a map of degree $1$, thus $m_{\Gamma',i} = 1$.

\item
Let $a\in \A_{1,f}$ as in \cref{def:constr}\cref{it:constr_f1}.
We start by observing that the vanishing set of the function
$z = (f-\epsilon)/g$ is contained in the piece $M_{a,1,+}$.
The vanishing set of $z$ is the Milnor fiber $F_f$ of $f$. The intersection
$F_f \cap \partial T_a$ consists of two parts, contained in neighbourhoods
around $E_w\cap\partial T_a$ and $E_a\cap\partial T_w$. By construction,
the former is not included in $M_a$. The latter is homologous to
a meridian around $E_w$ with multiplicity $m_w$. We can take this meridian
as $E_a \cap \partial T_a$.
Therefore, the dual multiplicities vanish on all vertices of $G_a$
except for $v_{a,1,+}$ and we have $n_{a,1,+} = m_a$.

It follows from the explicit calculations in \cref{it:pf_f1}
in the proof of \cref{thm:boundary} that the restriction of
$f-\epsilon$ to a fiber of $M_{a,0}$ has degree zero.
Indeed, in the coordinates $u,v$ introduced there for the polydisk
$T_a$, we have $f|_{T_a} = u^{m_a} v^{m_w}$ and a fiber
in $M_{a,0}$ is parametrized in these coordinates by
$[0,1]\to T_a$,
$t \mapsto (e^{-t \tilde m_w 2\pi i}, e^{t\tilde m_a 2\pi i})$.
Since $g$ vanishes with order $l_w$ along $E_w$, and does not vanish along
$E_a$, it follows that the multiplicity $m_{a,0}$ equals
$-l_w \tilde m_a$.

Now, the sequence $m_{a,i,+}$, $i=1,\ldots, s$ satisfies
\begin{equation} \label{eq:string}
\begin{array}{lclclcll}
              &-&b_{a,1,+} m_{a,1,+} &+& m_{a,2,+}   &=& n_{a,+,1},& \\
  m_{a,i-1,+} &-&b_{a,i,+} m_{a,i,+} &+& m_{a,i+1,+} &=& 0,  &
    i=2,\ldots, s-1, \\
  m_{a,s-1,+} &-&b_{a,s,+} m_{a,s,+} & &             &=& m_{a,0}. &
\end{array}
\end{equation}
The same equations are satisfied by the sequence
$(l_w - \tilde m_w) \mu_i - m_a \tilde \mu_i$, as is easily checked.
It follows that the two sequences coincide, since the matrix associated
with this system of linear equations is negative definite. A similar argument
proves the statement for the multiplicities $m_{a,i,-}$.

\item
Let $e$ be an edge in $G$ connecting $v_1$ and $v_2$ as in
\cref{def:constr}\cref{it:constr_bridge}.
We start by observing that $z$ does not vanish on $M_e$, and so all dual
multiplicities are vanish for the vertices of $G_e$.

Similarly as above,
we find that the map $f-\epsilon$ restricted to a fiber
$C_{e,0} \subset M_{e,0}$ has degree zero, and $g$ has degree
$\tilde m_2 l_{v_1} - \tilde m_1 l_{v_2}$. It follows that
$m_{e,0} = \tilde m_1 l_{v_2} - \tilde m_2 l_{v_1}$.
Now, similar reasoning as above determines the multiplicities
$m_{e,i,\pm}$. Namely, we have linear equations
\begin{equation} \label{eq:string2}
\begin{array}{lclclcll}
                &-&b_{e,1,\pm} m_{e,1,\pm} &+& m_{e,2,\pm}   &=& m_{e,0},& \\
  m_{e,i-1,\pm} &-&b_{e,i,\pm} m_{e,i,\pm} &+& m_{e,i+1,\pm} &=& 0,  &
    i=2,\ldots, s-1, \\
  m_{e,s-1,\pm} &-&b_{e,s,\pm} m_{e,s,\pm} & &             &=& m_{w_2,\pm}. &
\end{array}
\end{equation}
The result follows as soon as we determine the multiplicities
$m_{w_2,\pm}$:

\item
Let $w\in\W_2$. Above, we have determined that $C_{w,-}$ is a
meridian around $E_w$, small with respect to $\epsilon$ and having the
opposite orientation than the standard meridian. It follows that
$z = (f-\epsilon)/g$ restricted to $C_{w,-}$ has degree $l_w$, i.e.
$m_{w,-} = l_w$. Furthermore, $z$ does not vanish on $M_{w,-}$, thus
$n_{w,0} = 0$. 

On the other hand, $C_{w,+}$ is an oriented meridian around $E_w$, with
respect to which $\epsilon$ is chosen small. It follws that
$m_{w,+} = m_w - l_w$. Furthermore, the vanishing set of $z$ in $M_{w,+}$
is homologous to the strict transform of $f$, with multiplicities. Therefore,
each $a\in\A_{f,2}$ contributes $m_a$ to $n_{w,+}$, resulting in the
sum given.

\item
Let $a\in \A_{g,2}$ and set $w = w_a$.
Similarly as in \cref{it:pfm_G1}, we find that $z$ does not vanish on
$M_a$. Therefore, $n_{a,i} = 0$ for $i=0,\ldots, m_w$.
In \cref{it:pf_g2} of the proof of
\cref{thm:boundary} we found that $C_{a,0}$ is a meridian around $E_a$.
We can assume that the restriction of $f-\epsilon$
to such a meridian has degree $0$. It is also clear that $g$
restricts to a degree $-l_a$ map on such a fiber. It follows that
$m_{a,0} = -l_a$. 

For $i=1,\ldots, m_w$, the fiber $C_{a,i}$ is a small meridian around
$F_f$. It follows that $m_{a,i} = 1$.
\qedhere
\end{blist}
\end{proof}

\bibliography{bibliography}

\end{document}